\documentclass[12pt]{amsart}

\usepackage{amsmath}
\usepackage{amsfonts}
\usepackage{amsthm}

\usepackage[cmtip,arrow]{xy}
\usepackage{pb-diagram, pb-xy}
\usepackage{pb-diagram, pb-xy}

\newtheorem{theorem}{Theorem}
\newtheorem{corollary}{Corollary}
\newtheorem{lemma}{Lemma}
\newtheorem{proposition}{Proposition}
\newtheorem{definition}{Definition}

\newcommand{\C}{\mathcal{C}}

\newcommand{\F}{\mathcal{F}}

\renewcommand{\L}{\mathcal{L}}

\begin{document}
\title{Gorenstein flat preenvelopes and weakly Ding injective covers}
\author{Alina Iacob}

%\address{Department of Mathematics and Statistics, University of North Carolina at Wilmington,
%     Wilmington, North Carolina 28403 USA,
%     Email: iacoba@uncw.edu}
%

\maketitle

\begin{abstract}
 We consider a (left) coherent ring $R$. We prove that
 if the character module of every Ding injective (left) $R$-module is
 Gorenstein flat, then the class of Gorenstein flat (right) $R$-modules,
 $\mathcal{GF}$, is preenveloping. We show that this is the case when every
 injective (left) $R$-module has finite flat dimension. In particular, $\mathcal{GF}$
 is preenveloping over any Ding-Chen ring.\\
 The proofs use the class of weakly Ding injective (left) $R$-modules,
 $w\mathcal{DI}$. We show that, when $w\mathcal{DI}$ is closed under extensions, the
 following statements are equivalent:\\
 1. The character module of every Ding injective left $R$-module is
 a Gorenstein flat right R-module.\\
 2. The class of weakly Ding injective left $R$-modules is closed under
 direct limits.\\
 3. The class of weakly Ding injective modules is covering.\\
 The equivalent statements (1)-(3) imply that $\mathcal{GF}$ is preenveloping.
 \end{abstract}


\section{introduction}

Gorenstein homological algebra is the relative version of homological algebra that replaces the classical projective (injective, flat) resolutions with the Gorenstein projective (Gorenstein injective, Gorenstein flat) ones. But while in classical homological algebra the existence of the projective (injective, flat) resolutions over arbitrary rings is well known, things are a bit different when it comes to Gorenstein homological algebra. The existence of the Gorenstein injective envelopes over arbitrary rings was an open question for quite some time. It was settled in the affirmative by Saroch and Stovicek in [38]. The existence of the Gorenstein injective covers was also an open question. Very recently, we showed (in [27]) that the class of Gorenstein injectives is covering if and only if it is closed under direct limits if and only the ring is noetherian such that the character module of every Gorenstein injective (left) $R$-module is a Gorenstein flat (right) $R$-module.\\

We consider here another open question: When is the class of Gorenstein flat modules, $\mathcal{GF}$, preenveloping?  We recall that a right $R$-module
 $G$ is Gorenstein flat if it is a cycle of an exact complex of flat right $R$-modules $F$, such that $F \otimes I$ is still an exact complex for any injective
 $_RI$. It is known that $\mathcal{GF}$ is a Kaplansky class over any ring. This implies that $\mathcal{GF}$ is preenveloping if and only if it is closed under direct
 products (by [9], Theorem 2.5 and Remark 3). The closure of the class of Gorenstein flat (right) $R$-modules under direct products requires the
 ring to be (left) coherent, so this is the largest class of rings over which one can hope that GF is preenveloping. For this reason, throughout
 the paper, we assume that $R$ is a left coherent ring.\\

Our results use the class of weakly Ding injective modules, $w\mathcal{DI}$, its left orthogonal class, $^\bot w\mathcal{DI}$, as well as the class $(^\bot w\mathcal{DI})^\bot$. The Ding injective modules are the cycles of the exact complexes of injective modules that also stay exact when applying a functor $Hom(A,-)$ with $A$ any FP-injective module.  The weakly Ding injective modules are the cycles of the exact complexes of FP-injective modules that stay exact when applying a functor $Hom(A,-)$ with A any FP-injective module.\\
 Our main result gives a sufficient condition for the existence of the Gorenstein flat preenvelopes. More precisely, we show that if the character module of any Ding injective module is Gorenstein flat, then the class of Gorenstein flat right R-modules is preenveloping. \\
 As already mentioned, the proofs use the class of weakly Ding injective
 modules. We prove first (Proposition 1) that any weakly Ding injective module $M$ is a direct sum, $M = F \oplus D$, where $F$ is FP-injective, and $D$ is Ding injective. We also show (Theorem 3) that $(^\bot w\mathcal{DI},(^\bot \mathcal{DI}) ^\bot)$ is a complete hereditary pair. Proposition 7 shows that the following statements are equivalent:\\
 1. $(^\bot w\mathcal{DI})^\bot)$ is a precovering class.\\
 2. $(^\bot w\mathcal{DI})^\bot)$ is a covering class.\\
 3. $(^\bot w\mathcal{DI})^\bot)$ is closed under direct limits.\\
 Thus the class $(^\bot w\mathcal{DI})^\bot)$ satisfies the Enochs’ conjecture (it is covering if and only if it is closed under direct limits).

 Proposition 12 shows that $(^\bot w\mathcal{DI})^\bot)$ is a covering class when the character module of any Ding injective module is Gorenstein flat. Theorem 5 shows that in this case the class of Gorenstein flat right $R$-modules is closed under direct products, and therefore it is preenveloping. Proposition 13 proves that the character module of any Ding injective left $R$-module is Gorenstein flat over any left coherent ring $R$ such that every injective module has finite flat dimension. An immediate application is that the class of Gorenstein flat right $R$-modules, $\mathcal{GF}$, is preenveloping over any Ding-Chen ring. \\
 Theorem 6 shows that if $R$ is left coherent such that the class of weakly Ding injective left $R$-modules is closed under extensions, then the following statements are equivalent:\\
 1. The character module of every weakly Ding injective $R$-module is a Gorenstein flat right $R$-module.\\
 2. The character module of every Ding injective $R$-module is a Gorenstein flat right $R$-module.\\
 3. $(w\mathcal{DI},\mathcal{GF})$ is a duality pair.\\
 4. The class of weakly Ding injective modules is closed under direct limits.\\
 5. The class of weakly Ding injective modules is covering.\\

 The equivalent statements (1)- (5) imply the following statement:\\
 6. The class of Gorenstein flat right R-modules is preenveloping.

\section{preliminaries}
Throughout the paper $R$ will denote an associative ring with identity. Unless otherwise stated, by \emph{module} we mean \emph{left} $R$-module.

We will denote by $Inj$ the class of all injective modules.
We recall that an $R$-module $M$ is Gorenstein injective if there exists an exact and $Hom(Inj, -)$ exact complex of injective modules\\ $\textbf{I}= \ldots \rightarrow I_1 \rightarrow I_0 \rightarrow I_{-1} \rightarrow \ldots $ such that $M = Ker (I_0 \rightarrow I_{-1})$.\\
We will use the notation $\mathcal{GI}$ for the class of Gorenstein injective modules.\\

Given a class of $R$-modules $\mathcal{F}$, we will denote as usual by $\F^\bot$ the class of all $R$-modules $M$ such that $Ext^1(F,M)=0$ for every $F \in \F$.\\
The left orthogonal class of $\F$, denoted $^\bot \F$, is the class of all $_RN$ such that $Ext^1(N,F)=0$ for every $F \in \F$.\\

 The Ding injective modules were introduced in [5] where they were
 called Gorenstein FP-injective modules. They were later renamed
 Ding injective modules in [19]. We recall that an $R$-module $M$ is Ding injective if there exists an exact and $Hom(\mathcal{FI},-)$ exact complex of injective modules $E = \ldots \rightarrow E_1 \rightarrow E_0 \rightarrow E_{-1} \rightarrow \ldots$  such that
 $M =Ker(E_0\to E_{-1})$.\\
 We will use the notation $\mathcal{DI}$ for the class of Ding injective modules.
 The weakly Ding injective modules were introduced in [46].\\

\begin{definition} A module $G$ is weakly Ding injective if there is an exact
 complex of FP-injective modules $\textbf{E}$ such that $Hom(F,\textbf{E})$ is exact for
 any FP-injective module $F$ and such that $G = Z_0 \textbf{E}$.
\end{definition}

 We use $w\mathcal{DI}$ to denote the class of weakly Ding injective modules. It
 is immediate from the definitions that any Ding injective is weakly Ding
 injective, and that any FP-injective is also a weakly Ding injective
 module.\\

 Given a class of $R$-modules $\mathcal{F}$, we will denote as usual by $\mathcal{F}^\bot$ the class
 of all $R$-modules M such that $Ext^1(F,M) = 0$ for every $M \in \mathcal{F}$.\\
The left orthogonal class of F, denoted $^\bot \mathcal{F}$, is the class of all $_RN$ such
 that $Ext^1(N,F) = 0$ for every $F \in \mathcal{F}$.\\

We also recall the definitions for precovers, covers, and special precovers. \\
\begin{definition}
 Let $\mathcal{C}$ be a class of R-modules. A homomorphism $\phi :
 G \rightarrow M$ is a $\mathcal{C}$-precover of $M$ if $G \in \mathcal{C}$ and if for any module $G' \in \mathcal{C}$ and
 any $\phi' \in  Hom(G',M)$ there exists $u \in Hom(G',G)$ such that $\phi' = \phi u$.\\
 A $\mathcal{C}$-precover $\phi$ is said to be a cover if any $v \in  End_R(G)$ such that
$\phi v =\phi$ is an automorphism of $G$.\\
 A $\mathcal{C}$ precover $\phi$ is said to be special if $Ker(\phi)$ is in the right orthogonal
 class of $\mathcal{C}$ (i.e. if $Ext^1(G',Ker(\phi) = 0$ for all modules $G' \in  \mathcal{C}$).
\end{definition}

 The importance of (pre)covers comes from the fact that they allow
 defining resolutions: if the ring $R$ is such that every $R$-module $M$ has
 a $C$ precover then for every $M$ there exists a $Hom(\mathcal{C},-)$ exact complex
 $\ldots \rightarrow G_1 \rightarrow G_0 \rightarrow M \rightarrow 0$ with $G_0 \rightarrow M$ and $G_i \rightarrow Ker(G_{i-1} \rightarrow G_{i-2})$ being $\mathcal{C}$-precovers. Such a complex is called a $\mathcal{C}$ resolution of $M $; it is unique up to homotopy so it can be used to compute right derived functors of $Hom$.

We also use Gorenstein flat right $R$-modules. They are defined in terms of the tensor product:\\
\begin{definition}
A right $R$-module $G$ is Gorenstein flat if there exists an exact complex of flat right $R$-modules  ${\rm {\bf F }}= \ldots \rightarrow F_1 \rightarrow F_0 \rightarrow F_{-1} \rightarrow \ldots $ such that ${\rm {\bf F }} \otimes I$ is still exact for any injective module $I$, and such that $G = Ker (F_0 \rightarrow F_{-1})$.\\
Such a complex ${\rm {\bf F }}$ is called an \emph{F-totally acyclic complex}.
\end{definition}

We will use $\mathcal{GF}$ to denote the class of Gorenstein flat right $R$-modules.\\

\begin{definition} A Gorenstein flat preenvelope of a right $R$-module $M$ is a homomorphism $\phi : M \rightarrow G$ with $G \in \mathcal{ GF}$ such that for any
 $G' \in \mathcal{ GF}$ and for any $f \in  Hom(M,G')$ there is $u \in  Hom(G,G')$ with
 the property that $f = u \circ \phi$.
\end{definition}

 If R is a ring such that $\mathcal{GF}$ is preenveloping, then every right $R$
module $M$ has a right Gorenstein flat resolution (or a Gorenstein flat
 co-resolution). Such a co-resolution is a complex $0 \rightarrow M \rightarrow G^0 \rightarrow G^1 \rightarrow G^2 \rightarrow \ldots$ with $M \rightarrow G^0$ and $Coker(G_{i-1} \rightarrow G_i) \rightarrow G_{i+1}$ a Gorenstein flat preenvelope. Such a co-resolution is unique up to homotopy, so it can be used to compute derived functors of $Hom$.\\

 Cotorsion pairs will be used throughout the paper, so we recall that a pair $(\mathcal{L}, \mathcal{C})$ is a \emph{cotorsion pair} if $\mathcal{L} ^\bot = \mathcal{C}$ and $^\bot \mathcal{C} = \mathcal{L}$.\\
A cotorsion pair $(\mathcal{L}, \mathcal{C})$ is \emph{complete} if for every $_RM$ there exists exact sequences $ 0 \rightarrow C \rightarrow L \rightarrow M \rightarrow 0$ and $0 \rightarrow M \rightarrow C'\rightarrow L' \rightarrow 0$ with $C$, $C'$ in $\mathcal{C}$ and $L$, $L'$ in $\mathcal{L}$.\\
%A cotorsion pair $(\mathcal{L}, \mathcal{C})$ is hereditary if $Ext^i(L,C)=0$ for any $L \in \mathcal{L}$, any $C \in \mathcal{C}$, and any $i \ge 1$.\\

\begin{definition} (\cite{jrgr}, Definition 1.2.10)
 A cotorsion pair $(\L, \C)$ is called hereditary
if one of the following equivalent statements hold:
\begin{enumerate}
\item $\L$ is resolving, that is, $\L$ is closed under taking kernels of epimorphisms.
\item $\C$ is coresolving, that is, $\C$ is closed under taking cokernels of monomorphisms.
\item $Ext^i (F, C) = 0$ for any $F \in \F$ and $C\in \C$ and $i\geq 1$.
\end{enumerate}
\end{definition}

We will also use duality pairs. They were introduced by Holm and
 Jorgensen in [24]. We recall the definition:\\

 \begin{definition}. Let $\mathcal{C}$ be a class of $R$-modules, and let $\mathcal{L}$ be a class of
 right $R$-modules. The pair $(\mathcal{C},\mathcal{L})$ is said to be a duality pair if:\\
 1. $M \in \mathcal{C}$ if and only if $M^+ \in L$.\\
 2. the class $\mathcal{L}$ is closed under direct summands and under finite direct
 sums.
\end{definition}

 By [24] Theorem 3.10, if  $(\mathcal{C},\mathcal{L})$ is a duality pair, then $\mathcal{C}$ is closed under
 pure submodules, pure extensions, and pure quotients. If, moreover, $\mathcal{C}$ is also closed under direct sums then $\mathcal{C}$ is a covering class.\\

 Some results used throughout the paper:\\

\begin{theorem} ([22], Theorem 44) The pair$ (^\bot \mathcal{DI}, \mathcal{DI})$ is a complete
 hereditary cotorsion pair over any ring $R$ (in fact, this is a perfect
 cotorsion pair, i.e. $^\bot{\mathcal{DI}}$ is covering and $\mathcal{DI}$ is enveloping).
\end{theorem}

 \begin{theorem} ([46], Lemma 5.3.9) Let $R$ be a left coherent ring. If
 $M$ is a weakly Ding injective $R$-module, then $M = F \oplus D$  with $F$ an FP-injective module, and with $D \in  \mathcal{FI}^\bot$ .
\end{theorem}

 In fact, the proof of this result given in [46] gives more details on the  modules $F$ and $D$. Since the module $ $M is weakly Ding injective, there is an exact
 sequence $0 \rightarrow  M \rightarrow  F^0  \rightarrow K \rightarrow 0$  with $F^0$ FP-injective, and with $K$
 weakly Ding injective. The proof of Theorem 2 shows that $M = F \oplus H$ where $F$ is isomorphic to a direct summand of $F^0$, and $D$ is isomorphic
 to $H$, where $ 0 \rightarrow  H \rightarrow  G \rightarrow K \rightarrow 0$ is exact such that $G \rightarrow  K$ is an FP-injective cover.

\section{main results}

We show that over any coherent ring $R$, every weakly Ding injective module $M$ is a direct sum, $M = F \oplus D$ with $F$ FP-injective and with $D$ Ding injective.

\begin{lemma}
Let $R$ be a left coherent ring.
If $M = F \oplus I$ is a weakly Ding injective module, with $F$ FP-injective and with $I$ in $\mathcal{FI}^\bot$, then there is an exact and $Hom(\mathcal{FI}, -)$ exact complex $\ldots \rightarrow E_2 \rightarrow E_1 \rightarrow E_0 \rightarrow H \rightarrow 0$ with all $E_i$ injective modules.

\end{lemma}

\begin{proof}
By the proof of Theorem 2, $I$ is the kernel of an FP-injective cover of a weakly Ding injective module.
Since $M$ is weakly Ding injective, and since $M = I \oplus K$, it follows that the FP-injective cover of $M$ is surjective, of the form $E\oplus K \rightarrow I \oplus K \rightarrow 0$.\\
So there is an exact sequence $0 \rightarrow J \rightarrow E \rightarrow I \rightarrow 0$ with $E \rightarrow I$ an FP-injective cover, where $K \in w\mathcal{DI}$.\\
Both $I$ and $J$ are in $\mathcal{FI}^\bot$, so $E \in \mathcal{FI}^\bot \bigcap \mathcal{FI}$. Thus $E$ is an injective module.\\

By the definition of a weakly Ding injective module, there is an exact sequence $0 \rightarrow M_0 \rightarrow F_0 \rightarrow M \rightarrow 0$ with $F_0$ FP-injective and with $M_0$ weakly Ding injective.\\
By the proof of Theorem 2, we have that $F_0 = E \oplus K \oplus V$ and $M_0 = J \oplus V$, where $V$ is FP-injective, and $J \in \mathcal{FI}^\bot$. \\
Then (since $M_0 = J \oplus V$ is weakly Ding injective) there is an exact sequence $0 \rightarrow J_1 \rightarrow E_1 \oplus V \rightarrow J \oplus V \rightarrow 0$ where $E_1 \oplus V \rightarrow J \oplus V = M_0$ is an FP-injective cover, so $J_1 \in \mathcal{FI}^\bot$.  This gives an exact sequence $0 \rightarrow J_1 \rightarrow E_1 \rightarrow J \rightarrow 0$ with both $J, J_1 \in \mathcal{FI}^\bot$. Therefore $E_1 \in \mathcal{FI} \bigcap \mathcal{FI}^\bot = Inj$.\\
Thus we have an exact sequence $0 \rightarrow J_1 \rightarrow E_1 \rightarrow E \rightarrow I \rightarrow 0$ with $E,E_1$ injective modules, and with $J_1 \in \mathcal{FI}^\bot$.\\
Also, because $M_0 = J \oplus V$ is weakly Ding injective, there is an exact sequence $0 \rightarrow M_1 \rightarrow F_1 \rightarrow M_0  \rightarrow 0$. The proof of theorem 2 gives that $M_1 = J_1 \oplus T$ where $T$ is FP-injective. Then, the same argument as above shows that there is an exact sequence $0 \rightarrow J_2 \rightarrow E_2 \rightarrow J_1 \rightarrow 0$ with $J_2 \in \mathcal{FI}^\bot$, and with $E_2 \in \mathcal{FI} \bigcap \mathcal{FI}^\bot = Inj$.\\
Continuing, we obtain an exact complex $\ldots \rightarrow E_2 \rightarrow E_1 \rightarrow E \rightarrow I \rightarrow 0$ where all the $E_i$s are injective modules, and all cycles are in $\mathcal{FI}^\bot$.
\end{proof}

\begin{lemma}
Let $R$ be a left coherent ring.
If $M = F \oplus I$ is a weakly Ding injective module with $F$ FP-injective and with $I$ in $\mathcal{FI}^\bot$, then there is an exact and $Hom(\mathcal{FI}, -)$ exact complex $0 \rightarrow I \rightarrow E^0 \rightarrow E^1 \rightarrow E^2 \rightarrow \ldots$ with all $E^i$ injective modules.
\end{lemma}

\begin{proof}
By definition there is an exact sequence $0 \rightarrow M \rightarrow F^0 \rightarrow M^0 \rightarrow 0$ with $F$ FP-injective and with $M^0$ weakly Ding injective. By the proof of theorem 2, $M = I \oplus K$ where $K$ is FP-injective and $I \in \mathcal{FI}^\bot$.\\ More precisely, there is an exact sequence $0 \rightarrow I \rightarrow F \rightarrow M^0 \rightarrow 0$ with $F \rightarrow M^0 \rightarrow 0$ the FP-injective cover.\\
$M^0$ is weakly Ding injective, so there is an exact sequence $0 \rightarrow M^0 \rightarrow F^1 \rightarrow M^1 \rightarrow 0$ with $F^1$ FP-injective and with $M^1$ weakly Ding injective. As above, $M^0 = I^1 \oplus K^1$, $F^1 = K^1 \oplus F'$ where $0 \rightarrow I^1 \rightarrow F^1 \rightarrow M^1 \rightarrow 0$ is exact with $F^1 \rightarrow M^1 \rightarrow 0$ the FP-injective cover.\\

Since $M^0 = I^1 \oplus K^1$ with $K^1$ FP-injective, and $F \rightarrow M^0$ is the FP-injective cover, it follows that $F = E^1 \oplus K^1$.\\
So there is an exact sequence $0 \rightarrow I \rightarrow E^1 \oplus K^1 \rightarrow I^1 \oplus K^1 \rightarrow 0$. After factoring out the exact subcomplex $0 \rightarrow K^1 = K^1 \rightarrow 0$ we obtain an exact sequence $0 \rightarrow I \rightarrow E^1 \rightarrow I^1 \rightarrow 0$ with $I, I^1 \in \mathcal{FI}^\bot$ and with $E^1 \in \mathcal{FI} \bigcap \mathcal{FI}^\bot = Inj$.\\

Continuing we obtain an exact complex $0 \rightarrow I \rightarrow I^1 \rightarrow I^2 \rightarrow \ldots$ with all $E^j$ injective modules and with all cycles in $\mathcal{FI}^\bot$.
\end{proof}

\begin{proposition}
Let $R$ be a coherent ring. If $M$ is a weakly Ding injective $R$-module then $M = K \oplus I$ with $I$ Ding injective and with $K$ FP-injective.
\end{proposition}

\begin{proof} By Theorem 2, $M = K \oplus I$ with $K$ FP-injective and with $I \in \mathcal{FI}^\bot$.
By Lemma 1 and Lemma 2, there are exact complexes:
$$0 \rightarrow I \rightarrow E^1 \rightarrow E^2 \rightarrow \ldots$$
and $$\ldots \rightarrow E_2 \rightarrow E_1 \rightarrow E \rightarrow I \rightarrow 0$$

with all $E_j$ and $E^i$ injective modules, and with all cycles in $\mathcal{FI}^\bot$.\\
Pasting them together we obtain an exact and $Hom(\mathcal{FI}, -)$ exact complex of injective modules $\ldots \rightarrow E_2 \rightarrow E_1 \rightarrow E \rightarrow E^1 \rightarrow E^2 \rightarrow \ldots$, with $I$ being one of the cycles. \\
Thus $I$ is a Ding injective module.
\end{proof}

\begin{proposition}
Let $R$ be a left coherent ring. Then the left orthogonal class of $w \mathcal{DI}$ is $FP-Proj \bigcap (^\bot \mathcal{DI})$.
\end{proposition}

\begin{proof}
$"\supseteq"$ Let $T \in FP-Proj \bigcap (^\bot \mathcal{DI})$. If $M$ is a weakly Ding injective module, then $M = F \oplus H$ with $F$ an FP-injective module and with $H$ Ding injective. Then $Ext^1(T,M) \simeq Ext^1(T,F) \oplus Ext^1(T, H)$. Since $T$ is FP-projective, $Ext^1(T,F) =0$, and since $T$ is in $^\bot \mathcal{DI}$, we have that $Ext^1(T,H)=0$. It follows that $Ext^1(T, M)=0$.

$"\subseteq"$ We have $FP-Inj \subseteq w \mathcal{DI}$, which implies that $^\bot w \mathcal{DI} \subseteq {}^\bot FP-Inj = FP-Proj$.
Since $\mathcal{DI} \subseteq w\mathcal{DI}$, we also have that $^\bot w\mathcal{DI} \subseteq ^\bot\mathcal{DI} $. So $^\bot w \mathcal{DI} \subseteq FP-Proj \bigcap ^\bot \mathcal{DI}$.
\end{proof}

\begin{lemma}
Let $R$ be a coherent ring. Let $M \in (^\bot w \mathcal{DI})^\bot$. If $M$ is also in $^\bot \mathcal{DI}$, then $M$ is FP-injective.
\end{lemma}

\begin{proof}
Since $(FP-Proj, \mathcal{FI})$ is a complete cotorson pair (by \cite{trlifaj}, Theorem 3.8), there is an exact sequence $0 \rightarrow M \rightarrow A \rightarrow T \rightarrow 0$ with $A$ FP-injective, and with $T$ FP-projective.\\
Both $M$ and $A$ are in $^\bot \mathcal{DI}$, and this class is thick (by \cite{gil-iacob}, Lemma 22). It follows that $T$ is also in $^\bot \mathcal{DI}$.\\
Thus $T \in FP-Proj \bigcap ^\bot\mathcal{DI}= ^\bot w \mathcal{DI}$ (by Proposition 2), and $M \in (^\bot w \mathcal{DI})^\bot$. Therefore $Ext^1(T,M) =0$. Then $A \simeq M \oplus T$, which implies that $M$ is an FP-injective module.
\end{proof}

\begin{lemma}
Let $R$ be a left coherent ring. $_RM \in (^\bot w\mathcal{DI})^\bot$ if and only if there is an exact sequence $0 \rightarrow D \rightarrow F \rightarrow M \rightarrow 0$ with $D$ Ding injective and with $F$ FP-injective.
\end{lemma}

\begin{proof}
Assume that $_RM \in (^\bot w\mathcal{DI})^\bot$. Since $(^\bot \mathcal{DI}, \mathcal{DI})$ is a complete cotorsion pair, there is an exact sequence $0 \rightarrow D \rightarrow C \rightarrow M \rightarrow 0$ with $D \in \mathcal{DI} \subseteq (^\bot w \mathcal{DI})^\bot$, and with $C \in ^\bot \mathcal{DI}$.\\
Both modules $D$ and $M$ are in $(^\bot w \mathcal{DI})^\bot$, which implies that $C$ is also in $(^\bot w \mathcal{DI})^\bot$. But then $C$ is both in $^\bot \mathcal{DI}$ and in $(^\bot w \mathcal{DI})^\bot$, so, by Lemma 3, it is an FP-injective module.\\

For the converse, assume that there is an exact sequence $0 \rightarrow D \rightarrow F \rightarrow M \rightarrow 0$ with $D$ Ding injective and with $F$ FP-injective.\\
Let $T \in ^\bot w\mathcal{DI} = FP-Proj \bigcap (^\bot \mathcal{DI})$. We have an exact sequence: $0 = Ext^1 (T, F) \rightarrow Ext^1(T,M) \rightarrow Ext^2(T, D)=0$ (because $(^\bot \mathcal{DI}, \mathcal{DI})$ is a hereditary cotorsion pair, by [22]). So $Ext^1(T,M)=0$, and therefore $_RM \in (^\bot w\mathcal{DI})^\bot$.
\end{proof}

\begin{proposition}
Let $R$ be a coherent ring. If $M \in (^\bot w \mathcal{DI})^\bot$ then there is an exact and $Hom(\mathcal{FI},-)$ exact complex $\ldots \rightarrow F_2 \rightarrow F_1 \rightarrow F_0 \rightarrow M \rightarrow 0$ with each $F_j$ FP-injective.
\end{proposition}

\begin{proof}
By Lemma 4, there is an exact sequence $0 \rightarrow D \rightarrow C \rightarrow M \rightarrow 0$, with $D$ Ding injective and with $C$ an FP-injective module.\\
%Since $(^\bot \mathcal{DI}, \mathcal{DI})$ is a complete cotorsion pair, there is an exact sequence $0 \rightarrow D \rightarrow C \rightarrow M \rightarrow 0$ with $D \in \mathcal{DI} \subseteq (^\bot w \mathcal{DI})^\bot$, and with $C \in ^\bot \mathcal{DI}$.\\
%Both modules $D$ and $M$ are in $(^\bot w \mathcal{DI})^\bot$, which implies that $C$ is also in $(^\bot w \mathcal{DI})^\bot$. But then $C$ is both in $^\bot \mathcal{DI}$ and in $(^\bot w \mathcal{DI})^\bot$, so, by Lemma 3, it is an FP-injective module.\\
Since $D$ is Ding injective, there is an exact complex $\ldots \rightarrow F_2 \rightarrow F_1 \rightarrow D \rightarrow 0$ with all $F_i$ injective modules, and with all cycles Ding injective modules. Thus we obtain an exact complex $\ldots \rightarrow F_2 \rightarrow F_1 \rightarrow C \rightarrow M \rightarrow 0$ with $C$ and all $F_j$ FP-injective. Since all the ith cycles ($i \ge 1$) are Ding injective modules, the complex is $Hom(\mathcal{FI},-)$ exact.
\end{proof}

We show that if ($R$ is coherent and) the character module of every module in $(^\bot w \mathcal{DI})^\bot$ is Gorenstein flat, then $(^\bot w \mathcal{DI})^\bot$ is closed under direct limits, and it is a covering class.\\
Then we prove that with the same hypothesis ($R$ is coherent and the character module of every module in $(^\bot w \mathcal{DI})^\bot$ is Gorenstein flat), the class of Gorenstein flat modules is preenveloping. In particular, we show that this is the case when $R$ is a Ding-Chen ring.

We prove first that $(^\bot w\mathcal{DI}, (^\bot w \mathcal{DI})^\bot)$ is a complete hereditary pair.

\begin{proposition}
Let $R$ be a left coherent ring. Then $(^\bot w\mathcal{DI}, (^\bot w \mathcal{DI})^\bot)$ is a hereditary cotorsion pair.
\end{proposition}

\begin{proof}
It is known that $(^\bot w \mathcal{DI},  (^\bot w \mathcal{DI})^\bot)$ is a cotorsion pair.\\

We show that the pair is hereditary. Let $M$ be a weakly Ding injective module. Then $M = F \oplus D$ with $F$ an FP-injective module and with $D$ a Ding injective module. Let $T \in ^\bot w \mathcal{DI}$. By Proposition 2, $M \in FP-Proj \bigcap (^\bot \mathcal{DI})$.\\
Since $T \in FP-Proj$,  $F \in \mathcal{FI}$, and $R$ is a coherent ring (so, by \cite{kathy}, Proposition 4.2, $\mathcal{FI}$ is closed under cokernels of monomorphisms), we have that $Ext^i(T,F)=0$, for all $i \ge 1$.\\
Since $(^\bot \mathcal{DI}, \mathcal{DI})$ is a hereditary cotorsion pair (over any ring), we have that $Ext^i(T,D)=0$ for all $i \ge 1$. Thus $Ext^i(T,M) = Ext^i(T,F) \oplus Ext^i(T,D)=0$, for all $i \ge 1$.
\end{proof}

\begin{theorem}
Let $R$ be a coherent ring. Then $(^\bot w \mathcal{DI},  (^\bot w \mathcal{DI})^\bot)$ is a complete and hereditary cotorsion pair.
\end{theorem}

\begin{proof}
We prove that $^\bot w\mathcal{DI}$ is special precovering.\\
Let $M$ be a left $R$-module. Since $(^\bot \mathcal{DI}, \mathcal{DI})$ is a complete hereditary cotorsion pair, there is a short exact sequence $0 \rightarrow D \rightarrow X \rightarrow M \rightarrow 0$ with $X \in ^\bot \mathcal{DI}$, and with $D \in \mathcal{DI}$. Since $R$ is coherent, the pair $(\mathcal{FP}-Proj, \mathcal{FI})$ is complete, so there is an exact sequence $0 \rightarrow L \rightarrow T \rightarrow X \rightarrow 0$ with $T$ an FP-projective module, and with $L$ FP-injective.
We have that $L \in \mathcal{FI} \subseteq ^\bot \mathcal{DI}$, and $X \in ^\bot \mathcal{DI}$. Since $^\bot \mathcal{DI}$ is closed under extensions, it follows that $T \in \mathcal{FP}-Proj \cap ^\bot \mathcal{DI}$, that is, $T \in ^\bot w \mathcal{DI}$.\\
Form the pull back diagram\\

\[
\begin{diagram}
\node{}\node{0}\arrow{s}\node{0}\arrow{s}\\
\node{}\node{L}\arrow{s}\arrow{e,=}\node{L}\arrow{s}\\
\node{0}\arrow{e}\node{K}\arrow{s}\arrow{e}\node{T}\arrow{s}\arrow{e}\node{M}\arrow{s,=}\arrow{e}\node{0}\\
\node{0}\arrow{e}\node{D}\arrow{e}\node{X}\arrow{e}\node{M}\arrow{e}\node{0}
\end{diagram}
\]

The exact sequence $0 \rightarrow L \rightarrow K \rightarrow D \rightarrow 0$ with $L \in \mathcal{FI} \subseteq w\mathcal{DI} \subseteq (^\bot w\mathcal{DI})^\bot$ and $D \in \mathcal{DI} \subseteq w\mathcal{DI} \subseteq (^\bot w\mathcal{DI})^\bot$ gives that $K \in \subseteq (^\bot w\mathcal{DI})^\bot$.\\
The exact sequence $0 \rightarrow K \rightarrow T \rightarrow M \rightarrow 0$ with $T \in ^\bot w\mathcal{DI}$ and with $K \in \subseteq (^\bot w\mathcal{DI})^\bot$ shows that $T \rightarrow M$ is a special $ ^\bot w\mathcal{DI}$ precover.\\

Thus $(^\bot w \mathcal{DI}, (^\bot w\mathcal{DI})^\bot)$ is a complete hereditary cotorsion pair.\\

\end{proof}

\begin{lemma}
Let $R$ be a left coherent ring. Then $(^\bot w\mathcal{DI})^\bot$ is closed under direct summands and under cokernels of monomorphisms.
\end{lemma}
\begin{proof}
Since $(^\bot w\mathcal{DI})^\bot$ is the right half of a hereditary cotorsion pair (by Proposition 4) it follows that $w\mathcal{DI}$ is closed under cokernels of monomorphisms. Also, as the right half of a cotorsion pair, $w\mathcal{DI}$ is closed under direct summands.
\end{proof}

We can prove now that (for a coherent ring $R$) the class  $(^\bot w\mathcal{DI})^\bot$ is special precovering if and only if it is closed under direct limits. In particular,  $(^\bot w\mathcal{DI})^\bot$ is covering if and only if it is closed under direct limits. The proof uses the following result (\cite{amuc25}, Proposition 3):\\
\begin{proposition}
Let $\mathcal{W}$ be a class of modules that is closed under direct summands, under taking cokernels of pure monomorphisms, and under pure transfinite extensions. Then $\mathcal{W}$ is closed under direct limits.
\end{proposition}

We also use the following result:\\

\begin{proposition}
If every $R$-module has a special $(^\bot w\mathcal{DI})^\bot$-precover,  then  the  class  $(^\bot w\mathcal{DI})^\bot$ is  closed  under transfinite extensions.
\end{proposition}

\begin{proof}
For simplicity, let $\mathcal{A}$ denote the class  $(^\bot w\mathcal{DI})^\bot$ .\\
Let $ (G_{\alpha}, \alpha \le \lambda)$ be a direct system of monomorphisms with each $G_{\alpha} \in \mathcal{A}$ and let $G $ be the direct limit of the system.  Since for each $\alpha$, we have that $G_{\alpha} \in {}^\bot (\mathcal{A}^\bot)$, it follows that $G \in {}^\bot (\mathcal{A}^\bot)$ (by Eklof Lemma).  \\
The class $\mathcal{A}$ is precovering and closed under direct summands, so $\mathcal{A}$ is closed under arbitrary direct sums (\cite{trlifaj}, Lemma 9.14).  Thus $\oplus_{\alpha \le \lambda} G_{\alpha} \in \mathcal{A}$.\\
Since  there  is  a  short  exact  sequence  $0 \rightarrow K \rightarrow  \oplus_{\alpha \le \lambda} G_{\alpha} \rightarrow G \rightarrow 0$  with $\oplus_{\alpha \le \lambda} G_{\alpha} \in \mathcal{A}$, it follows that any $\mathcal{A}$-precover of $G$ has to be surjective.\\
The  class $\mathcal{A}$ is  special  precovering,  so  there  is  an  exact  sequence  $0 \rightarrow A \rightarrow \overline{G} \rightarrow G \rightarrow 0$ with $A \in \mathcal{A}^\bot$ and with $\overline{G} \in \mathcal{A}$.   But  $G \in {}^\bot (\mathcal{A}^\bot)$, so we have that $Ext^1(G;A) = 0$.  Thus $G$ is a direct summand of $\overline{G}$, therefore $G \in \mathcal{A}$.

\end{proof}

\begin{proposition}
Let $R$ be a left coherent ring. Then the class  $(^\bot w\mathcal{DI})^\bot$ is special precovering if and only if it is closed under direct limits.
\end{proposition}

\begin{proof}
Assume that $(^\bot w\mathcal{DI})^\bot$ is closed under direct limits. Since  $(^\bot w\mathcal{DI})^\bot$ is the right half of a complete hereditary cotorsion pair, it follows that it is a definable class ([39]). Thus  $(^\bot w\mathcal{DI})^\bot$ is closed under pure submodules. But then, since  $(^\bot w\mathcal{DI})^\bot$ is also closed under cokernels of monomorphisms, it follows that  $(^\bot w\mathcal{DI})^\bot$ is closed under pure quotients as well. By \cite{purity}, a class of modules that is closed under pure quotients and direct sums is a covering class. Thus $(^\bot w\mathcal{DI})^\bot$ is covering. Since $(^\bot w\mathcal{DI})^\bot$ is covering and closed under direct summands, it is a special precovering class.\\
Conversely, assume that $(^\bot w\mathcal{DI})^\bot$ is a special precovering class. Since  $(^\bot w\mathcal{DI})^\bot$ is closed under transfinite extensions (by Proposition 6), and it is also closed under direct summands, under cokernels of monomorphisms, and under pure transfinite extensions, it follows (Proposition 5) that  $(^\bot w\mathcal{DI})^\bot$ is closed under direct limits.
\end{proof}

 \begin{corollary} The class $(^\bot w\mathcal{DI})^\bot$ is covering if and only if it is closed
 under direct limits.
\end{corollary}

 \begin{proof} If $(^\bot w\mathcal{DI})^\bot$ is closed under direct limits then, by the proof of
 Proposition 7, it is covering.\\
 Conversely, assume that $(^\bot w\mathcal{DI})^\bot$ is covering. Since $(^\bot w\mathcal{DI})^\bot$ is closed
 under direct dummands, it is also special precovering. By Proposition
 7, it is closed under direct limits.
\end{proof}

\begin{proposition}
Let $R$ be a coherent ring. If the character module of any $_RM \in (^\bot w\mathcal{DI})^\bot$ is a Gorenstein flat right $R$-module, then $((^\bot w\mathcal{DI})^\bot, \mathcal{GF})$ is a duality pair.
\end{proposition}

\begin{proof}
Since the class of Gorenstein flat modules is closed under direct summands and under direct sums, it suffices to show that $K$ is weakly Ding injective if and only if its character module, $K^+$, is Gorenstein flat. One implication is a hypothesis we made on the ring.\\
Conversely, assume that $K^+$ is Gorenstein flat. The pair $(^\bot \mathcal{DI}, \mathcal{DI})$ is complete, so there is an exact sequence $0 \rightarrow D \rightarrow L \rightarrow K \rightarrow 0$ with $D$ Ding injective and with $L \in ^\bot \mathcal{DI}$. Since $Q/Z$ is an injective $Z$-module, the sequence $0 \rightarrow K^+ \rightarrow L^+ \rightarrow D^+ \rightarrow 0$ is exact. By hypothesis we have on the ring, $D^+$ is Gorenstein flat. Since $K^+$ is also Gorenstein flat, it follows that $L^+$ is Gorenstein flat. But $L^+ \in \mathcal{GF}^\bot$ (since for any Gorenstein flat module $Y$ we have $Ext^1(Y, L^+) \simeq Ext^1(L, Y^+)=0$ because $Y^+$ is Ding injective). So $L^+ \in \mathcal{GF} \cap \mathcal{GF}^\bot = Flat \cap Cotorsion$ (\cite{RIP}). Since $L^+$ is flat and $R$ is coherent, it follows that $L$ is FP-injective. Thus we have an exact sequence $0 \rightarrow D \rightarrow L \rightarrow K \rightarrow 0$ with $D$ Ding injective and with $L$ FP-injective.\\
Let $A \in ^\bot w\mathcal{DI} = \mathcal{FP}-Proj \cap ^\bot\mathcal{DI}$. The exact sequence $0 \rightarrow D \rightarrow L \rightarrow K \rightarrow 0$ gives an exact sequence $0 =Ext^1 (A,L) \rightarrow Ext^1(A,K) \rightarrow Ext^2(A, D)=0$ (since the pair $(^\bot \mathcal{DI}, \mathcal{DI})$ is hereditary). Thus $Ext^1(A,K) = 0$ for all $A \in ^\bot w\mathcal{DI}$, and therefore $K$ is in $(^\bot w\mathcal{DI})^\bot$.
\end{proof}

\begin{proposition}
Let $R$ be a left coherent ring. If the character module of any $_RM \in (^\bot w\mathcal{DI})^\bot$ is Gorenstein flat, then the class $(^\bot w\mathcal{DI})^\bot$ is closed under direct limits. In particular, $(^\bot w\mathcal{DI})^\bot$ is a covering class in this case.
\end{proposition}

\begin{proof}
By Proposition 8, $((^\bot w\mathcal{DI})^\bot, \mathcal{GF})$ is a duality pair. Therefore $(^\bot w\mathcal{DI})^\bot$ is closed under pure submodules. Since $(^\bot w\mathcal{DI})^\bot$ is closed under direct products, and since the direct sum of a family of modules is a pure submodule of their direct product, it follows that $(^\bot w\mathcal{DI})^\bot$ is closed under direct sums.\\
By \cite{cotpairs}, Theorem 3.1, $(^\bot w\mathcal{DI})^\bot$ is also closed under pure quotients. Since for any directed system, the direct limit is a pure quotient of the direct sum of the module, we have that $(^\bot w\mathcal{DI})^\bot$ is closed under direct limits.\\
By \cite{purity}, Theorem 2, a class of modules that is closed under pure quotients and direct sums is a covering class. Thus $(^\bot w\mathcal{DI})^\bot$ is covering.
\end{proof}

\section{Gorenstein flat preenvelopes}

It is known that the class of Gorenstein flat modules, $\mathcal{GF}$, is a Kaplansky class. It is also known that a Kaplansky class is preenveloping if and only if is closed under direct products. We show that $\mathcal{GF}$ is closed under direct products over any left coherent ring $R$ such that the character of every module in the class $(^\bot w\mathcal{DI})^\bot$ is a Gorenstein flat right $R$-module.

\begin{theorem}
Let $R$ be a left coherent ring. If the character module of any $_RM \in (^\bot w\mathcal{DI})^\bot$ is a right Gorenstein flat module, then the class of Gorenstein flat right $R$-modules is closed under direct products.
\end{theorem}

\begin{proof}
By Proposition 9, the class $(^\bot w\mathcal{DI})^\bot$ is covering in this case, so it is closed under direct sums.\\

Let $(G_i)$ be a family of Gorenstein flat right $R$-modules. Each $G_i ^+$ is a Ding injective module, hence weakly Ding injective, therefore in $(^\bot w\mathcal{DI})^\bot$. By the above, $\oplus_i G_i ^+$ is in $(^\bot w\mathcal{DI})^\bot$. By hypothesis, $(\oplus_i G_i ^+)^+$ is a Gorenstein flat module. Therefore $\prod G_i^{++} \simeq (\oplus_i G_i ^+)^+$ is Gorenstein flat.\\

For each $i \in I$ we have a pure exact sequence: $0 \rightarrow G_i \rightarrow G_i ^{++} \rightarrow Y_i \rightarrow 0$. Thus
we have an exact sequence $0 \rightarrow \prod_i G_i \rightarrow \prod_i G_i^{++} \rightarrow \prod_iY_i \rightarrow 0$.\\

We show that this sequence is also pure exact.\\

Let $A_R$ be a finitely presented $R$-module. Then for each $i$ the sequence $0 \rightarrow Hom(A, G_i) \rightarrow Hom(A, G_i^{++}) \rightarrow Hom(A, Y_i) \rightarrow 0$ is still exact. So we have an exact sequence\\

$$0 \rightarrow \prod_i Hom(A, G_i) \rightarrow \prod_i Hom(A, G_i^{++}) \rightarrow \prod_i Hom(A, Y_i) \rightarrow 0$$

Since $\prod_i Hom(A, G_i) \simeq Hom (A, \prod_i G_i)$, $Hom (A, \prod_i G_i^{++})$, and $\prod_i Hom(A, Y_i) \simeq Hom (A, \prod_i Y_i)$, it follows that, for any finitely presented module $A_R$, the sequence $0 \rightarrow Hom(A, \prod_i G_i) \rightarrow Hom (A, \prod_i G_i ^{++}) \rightarrow Hom (A, \prod_i Y_i) \rightarrow 0$ is still exact.\\

So $\prod_i G_i$ is a pure submodule of the Gorenstein flat module $\prod_i G_i^{++}$. Therefore $(\prod_i G_i)^+$ is a direct summand of the Ding injective module $(\prod_i G_i^{++})^+$, so $(\prod G_i)^+$ is Ding injective, and so, Gorenstein injective. By \cite{holgordim}, Theorem 3.6, since $R$ is left coherent, this implies that $\prod_i G_i$ is Gorenstein flat.
\end{proof}

\begin{corollary}
Let $R$ be a left coherent ring. If the character module of any $_RM \in (^\bot w\mathcal{DI})^\bot$ is a Gorenstein flat right $R$-module, then the class of Gorenstein flat right $R$-modules is preenveloping.
\end{corollary}

\begin{proposition}
If $R$ is a Ding-Chen ring, then the class of Gorenstein flat right $R$-modules is preenveloping.
\end{proposition}
\begin{proof}
It suffices to show that $M^+ \in \mathcal{GF}$, for any $_RM \in (^\bot w\mathcal{DI})^\bot$.\\
Let $_RM \in (^\bot w\mathcal{DI})^\bot$. By Lemma 4, there is an exact sequence $0 \rightarrow D \rightarrow L \rightarrow M \rightarrow 0$ with $D$ Ding injective and with $L$ FP-injective. This implies that there is an exact and $Hom(\mathcal{FI},-)$ exact complex $\ldots \rightarrow E_2 \rightarrow E_1 \rightarrow L \rightarrow M \rightarrow 0$. By pasting this with a usual injective resolution of $M$, we obtain an exact complex of FP-injective modules $\textbf{F}$ such that $M = Z_0(\textbf{F})$.\\
Then $M^+ = Z_0 (F^+)$ with $F^+$ an exact complex of flat modules. By \cite{gil19}, $M^+$ is Gorenstein flat.\\
By Corollary 2, the class of Gorenstein flat right $R$-modules is preenveloping.
\end{proof}

We show that the condition that $M^+$ is Gorenstein flat for any $_RM \in (^\bot w\mathcal{DI})^\bot$ is equivalent to the condition that the character of any Ding injective left $R$-module being Gorenstein flat.

\begin{proposition}
Let $R$ be a left coherent ring. The character module of every $_RM \in (^\bot w\mathcal{DI})^\bot$ is a Gorenstein flat right $R$-module if and only if the character module of any Ding injective left $R$-module is Gorenstein flat.
\end{proposition}

\begin{proof}
One implication is immediate. We have $\mathcal{DI} \subseteq (^\bot w\mathcal{DI})^\bot$. Therefore, if the character module of every $_RM \in (^\bot w\mathcal{DI})^\bot$ is Gorenstein flat, then $D^+ \in \mathcal{GF}$, for any Ding injective module $D$.\\

Conversely, assume that $D^+ \in \mathcal{GF}$, for any Ding injective module $D$.\\
Let $_RM \in (^\bot w\mathcal{DI})^\bot$. By Lemma 4, there is an exact sequence $0 \rightarrow D \rightarrow F \rightarrow M \rightarrow 0$ with $D$ Ding injective, and with $F$ an FP-injective module. This gives the exact sequence $0 \rightarrow M^+ \rightarrow F^+ \rightarrow D^+ \rightarrow 0$. The ring is coherent, so $F^+$ is flat, therefore Gorenstein flat. By hypothesis $D^+$ is Gorenstein flat. Since $\mathcal{GF}$ is closed under kernels of epimorphisms, it follows that $M^+$ is Gorenstein flat.
\end{proof}

Using Proposition 11, we obtain:\\

\begin{proposition}
 Let $R$ be a left coherent ring. If the character module
 of any Ding injective module is Gorenstein flat, then the class $(^\bot w\mathcal{DI})^\bot$
 is closed under direct limits. In particular, $(^\bot w\mathcal{DI})^\bot$ is a covering class
 in this case.
\end{proposition}

\begin{proof}
 This follows from Proposition 9 and Proposition 11.
\end{proof}

\begin{theorem}
Let $R$ be a left coherent ring. If the character module of every Ding injective left $R$-module is Gorenstein flat, then the class of Gorenstein flat right $R$-modules is closed under direct products.
%In particular, the class of Gorenstein flat right $R$-modules is preenveloping in this case.
\end{theorem}

\begin{proof}
The result follows from Theorem 4 and Proposition 11.
\end{proof}

\begin{corollary}
Let $R$ be a left coherent ring. If the character module of every Ding injective left $R$-module is Gorenstein flat, then the class of Gorenstein flat right $R$-modules is preenveloping.
\end{corollary}

\begin{proof}
The result follows from Corollary 2 and Proposition 11.
\end{proof}

 The next result generalizes Proposition 10. The proof uses the following
 lemma.\\

\begin{lemma} Let R be a left coherent ring such that every injective left R
module has finite flat dimension. Then a right R-module is Gorenstein
 flat if and only if it is a cycle of an exact complex of flat right R
modules.
\end{lemma}

\begin{proof} One implication is immediate by the definition of a Gorenstein
 flat module.\\
 For the converse, assume that every injective left R-module I has finite
 flat dimension.
 Let $\textbf{F} = \ldots \rightarrow F_1 \rightarrow F_0 \rightarrow F_{-1} \rightarrow \ldots$ be an exact complex of flat right
 R-modules.
 We show that $\textbf{F}\otimes I$ is exact, for any injective $_RI$.\\

 Proof by induction on $n = flat.dim.I.$\\
 I. If n = 0 then I is flat, so $\textbf{F}\otimes I$ is exact.

 II. $n -1 \rightarrow  n$. There is an exact sequence $0 \rightarrow  L \rightarrow T \rightarrow I \rightarrow 0$ with
 $T$ flat and with $flat.dimL = n-1$.\\
 Since each $F_n$ is flat, we have an exact sequence $0 \rightarrow F_n \otimes L \rightarrow F_n \otimes T \rightarrow F_n \otimes I \rightarrow 0$. Therefore we have an exact sequence of
 complexes $0 \to \textbf{F}\otimes L \to \textbf{F}\otimes T \to\textbf{F}\otimes I \rightarrow 0$. The module $T$ is flat, so
 $\textbf{F}\otimes T$ is exact. By induction hypothesis, $\textbf{F}\otimes L$ is also exact. It follows
 that $\textbf{F}\otimes I$ is an exact complex.\\
 Thus the cycles of any exact complex of flat right R-modules are Gorenstein flat.

\end{proof}

\begin{proposition}
Let $R$ be a left coherent ring. If every injective left
 $R$-module has finite flat dimension, then the character module of any
 Ding injective left R-module is Gorenstein flat.
\end{proposition}

\begin{proof}
 Let $D$ be Ding injective. Then there is an exact complex $\textbf{E}$ of
 injective left $R$-modules, such that $D = Z_0(\textbf{E})$. Then $D^+ = Z_0(\textbf{E}^+)$. By
 Lemma 6, $D^+$ is Gorenstein flat.
\end{proof}

\begin{corollary} Let $R$ be a left coherent ring. If every injective left $R$-module has finite flat dimension, then the class of Gorenstein flat right
 R-modules, $\mathcal{GF}$, is preenveloping.
\end{corollary}

 \begin{proof} The result follows from Proposition 13 and Corollary 3.
\end{proof}

 Our next result uses the following:\\

\begin{lemma} Let $R$ be a left coherent ring, let $M$ be a right $R$-module. If $M^+$ is weakly Ding injective then $M$ is Gorenstein flat.
\end{lemma}

\begin{proof}
 Since $M^+ \in  w\mathcal{DI}$ there is an exact sequence $0 \rightarrow Z \rightarrow I \rightarrow M^+ \rightarrow 0$ with $I$ FP-injective. This gives an injective map $M^{++} \rightarrow I^+$ with $I^+$ flat. The canonical map $M \rightarrow M^{++}$ is also injective, so there is an injective morphism $M \rightarrow I^+$, with $I^+$ flat. It follows that any
 flat preenvelope $M \xrightarrow{\phi} F_{-1}$ is injective.\\
 Consider the exact sequence $0 \rightarrow M \xrightarrow{\phi} F_{-1} \rightarrow C_{-1} \rightarrow 0$ with $\phi$ a flat
 preenvelope. Then the sequence $0 \rightarrow C^+_{-1} \rightarrow F^+_{-1} \rightarrow M^+ \rightarrow 0$ is exact
 with $M^+$ weakly Ding injective and with $F^+ _{-1}$ injective.  We show that $C^+_{-1}\in \mathcal{ FI}^\bot$. This is equivalent to showing that, for any
 FP-injective module $T$, the sequence $Hom(T,F^+_{-1}) \rightarrow Hom(T,M^+) \rightarrow 0$ is exact (because $F^+ _{-1}$ is injective, so $Ext^1(T,F^+ _{-1}) = 0$).

 Consider the commutative diagram\\

\[
\begin{diagram}
\node{Hom(T,{F_{-1}}^+)}\arrow{s,r}{\simeq}\arrow{e,t}{l}\node{Hom(T,M^+)}\arrow{s,r}{\simeq}\\
\node{Hom(F_{-1}, T^+)}\arrow{e,t}{g}\node{Hom(M,T^+)}\arrow{e}\node{0}
\end{diagram}
\]\\

 Since $M \to F_{-1}$ is a flat preenvelope and $T^+$ is flat, we have that
 $Hom(F_{-1},T^+) \rightarrow  Hom(M,T^+) \rightarrow  0$ is exact.
 Thus the sequence $Hom(T,F^+_{-1}) \rightarrow  Hom(T,M^+) \rightarrow 0$ is exact. We
 also have an exact sequence $Hom(T,F^+_{ -1}) \rightarrow Hom(T,M^+) \rightarrow Ext^1(T,C^+ _{-1}) \rightarrow Ext^1(T,F^+_{-1}) = 0$ (since $F^+_{-1}$ is injective). It follows that $Ext^1(T,C^+_{-1}) =
 0$ for any FP-injective module T. Thus $ C^+_{-1}$ is in $\mathcal{FI}^\bot$, and therefore $F^+_{-1} \rightarrow  M^+$is a special FP-injective precover of $M^+$.\\

 By Proposition 1,  $M^+ = F \oplus K$ with $F$ FP-injective, and with $K$ Ding
 injective.\\
 Since $K$ is Ding injective, there is an exact sequence $0 \rightarrow K' \rightarrow E \rightarrow
 K \rightarrow 0$ with $E$ injective and with $K'$ Ding injective. So $E \rightarrow K$ is a
 special FP-injective precover of $K$. Therefore the FP-injective cover of
 K is of the form $V \rightarrow K$ where $V$ is a direct summand of $E$, so $V$ is
 injective.  Also, if $W = Ker(V \rightarrow K)$, then $W$ is a direct summand of
 $K'$, so it is Ding injective.\\
 Thus an FP-injective cover of $M^+$ is of the form $F \oplus V$ where $V \rightarrow K$
 is the FP-injective cover, $V$ is an injective module, and $W = Ker(V \rightarrow K$) is Ding injective.\\
 Since $0 \rightarrow W \rightarrow F \oplus V \rightarrow F \oplus K\rightarrow 0$ is exact, with $F \oplus V \rightarrow F \oplus K$
 the FP-injective cover, and $0 \rightarrow C^+_{-1} \rightarrow F^+_{-1} \rightarrow F \oplus K \rightarrow 0$ is exact
 with $F^+_{-1} \rightarrow F \oplus K$ a special FP-injective precover, it follows that
 $F^+_{-1} = F \oplus V \oplus U$, and $C^+_{-1} = W  \oplus U$. As a direct summand of the
 injective module $F^+_{-1}$,  $U$ is injective.\\
 Thus $C^+_{-1} = W  \oplus U$ is Ding injective, therefore Gorenstein injective.
 Since $R$ is coherent, it follows ([23]) that $C_{-1}$ is Gorenstein flat.
 The exact sequence $0 \rightarrow M \xrightarrow{\phi} F_{-1} \rightarrow C_{-1} \rightarrow 0$  with $F_{-1 }$ flat and with
 $C_{-1}$  Gorenstein flat gives that M is Gorenstein flat.

\end{proof}

\begin{proposition}
 Let $R$ be a coherent ring. The following statements
 are equivalent:\\
 (1) $w\mathcal{DI} = (^\bot w\mathcal{DI})^\bot$.\\
 (2) The class of weakly Ding injective modules is closed under extensions.
\end{proposition}

\begin{proof}
 ``(1) $\Rightarrow$ (2)'' A right orthogonal class is always closed under extensions.\\

 ``(2) $\Rightarrow$ (1)'' Let $M \in (^\bot w\mathcal{DI})^\bot$. Then there is an exact sequence
 $0 \rightarrow D'  \rightarrow F \rightarrow M \rightarrow 0$ with $F$ an FP-injective module, and with $D'$
 Ding injective (by Lemma 4).\\
 There is also an exact sequence $0 \rightarrow D' \rightarrow E \rightarrow D \rightarrow 0$ with $E$ injective and with $D'$ Ding injective. Since $E$ is injective, we have a commutative diagram with exact rows and columns:\\

\[
\begin{diagram}
\node{0}\arrow{e}\node{D'}\arrow{s,=}\arrow{e}\node{F}\arrow{s,r}{u}\arrow{e,t}{l}\node{M}\arrow{s,r}{h}\arrow{e}\node{0}\\
\node{0}\arrow{e}\node{D'}\arrow{e}\node{E}\arrow{e,t}{g}\node{D}\arrow{e}\node{0}
\end{diagram}
\]\\

 Thus we have a map of complexes $f : U \rightarrow V$, where $U = 0 \rightarrow D' \rightarrow F \rightarrow M \rightarrow 0$, and $V =0 \rightarrow D'\rightarrow E \rightarrow D\rightarrow 0$. Therefore we have an
 exact sequence of complexes $0 \rightarrow V \rightarrow c(f) \rightarrow  U[1] \rightarrow 0$.
 The complex $ c(f)$ has the exact subcomplex
 $\overline{D'} =0\rightarrow D' \xrightarrow{1_{D'}} D' \rightarrow 0$.  It follows that the complex $c(f)/{\overline{D'}} = 0 \to F \rightarrow E \oplus M \rightarrow D \rightarrow 0$ is
 also exact.\\
 Both $F$ and $D$ are weakly Ding injective modules, and $0 \rightarrow F \rightarrow E \oplus M \rightarrow D \rightarrow 0$ is exact. By (2), the module $M \oplus$ E is also weakly
 Ding injective. Thus there is an exact and $Hom(\mathcal{FI},-)$ exact sequence
$ 0 \rightarrow M \oplus E \rightarrow A^0\rightarrow W \rightarrow 0$ with $A^0$ FP-injective and with $W$ weakly Ding injective.

Consider the pushout diagram:\\

\[
\begin{diagram}
\node{}\node{}\node{0}\arrow{s}\node{0}\arrow{s}\\
\node{0}\arrow{e}\node{E}\arrow{s,=}\arrow{e}\node{M \oplus E}\arrow{s}\arrow{e}\node{M}\arrow{s}\arrow{e}\node{0}\\
\node{0}\arrow{e}\node{E}\arrow{e}\node{A^0}\arrow{s}\arrow{e}\node{C}\arrow{s}\arrow{e}\node{0}\\
\node{}\node{}\node{W}\arrow{s}\arrow{e,=}\node{W}\arrow{s}\\
\node{}\node{}\node{0}\node{0}
\end{diagram}
\]

So we have an exact sequence of exact complexes $0 \rightarrow X \rightarrow  Y \rightarrow Z \rightarrow 0$, where $X$, $Y$ and $Z$ are the columns in the diagram above:\\
 $X = 0 \rightarrow E \xrightarrow{1_E}  E \rightarrow 0$, $Y = 0 \rightarrow M \oplus E \rightarrow A^0 \rightarrow W \rightarrow 0$, and
 $Z =0 \rightarrow M \rightarrow C\rightarrow W \rightarrow 0$.\\
 Since $0 \rightarrow E \rightarrow A^0 \rightarrow C \rightarrow 0$ is exact with $E$ injective and $A^0$ FP-injective, it follows that $C$ is FP-injective.\\
 Since $E$ is injective, each row of the commutative diagram above is split exact. Therefore we have that, for every FP-injective module $I$, the sequence $0 \rightarrow Hom(I,X_n) \rightarrow  Hom(I,Y_n) \rightarrow  Hom(I,Z_n) \rightarrow 0$ is exact, for all n. Thus we have an exact sequence of complexes $0 \rightarrow Hom(I,X) \rightarrow Hom(I,Y) \rightarrow Hom(I,Z) \rightarrow 0$. Both $Hom(I,X)$ and $Hom(I,Y)$ are exact complexes, so $Hom(I,Z)$ is also exact, i.e. the complex $0 \rightarrow M \rightarrow C \rightarrow W \rightarrow 0$ is exact and $Hom(\mathcal{FI},-)$ exact.\\
 Since $W$ is weakly Ding injective there is an exact and $Hom(\mathcal{FI},-) $ exact complex $0 \rightarrow W \rightarrow A^1 \rightarrow A^2 \rightarrow \ldots$, with all $A^i$ FP-injective
 modules. So the complex $0 \rightarrow M \rightarrow C \rightarrow A^1 \rightarrow \ldots$ is exact and
 $Hom(\mathcal{FI},-)$ exact. Since $M \in (^\bot w\mathcal{DI})^\bot$, there is also an exact and $Hom(\mathcal{FI},-)$ exact complex $\ldots \rightarrow A^1 \rightarrow A^0 \rightarrow M \rightarrow 0$, with all $A^j$
 FP-injective modules (by Proposition 3). Pasting them together we  obtain an exact and $Hom(\mathcal{FI},-)$ exact complex of FP-injective modules A, with $M = Z_0(A)$. Thus $(^\bot w\mathcal{DI})^\bot\subseteq  w\mathcal{DI}$. The other inclusion
 always holds, so we have that $w\mathcal{DI} = (^\bot w\mathcal{DI})^\bot$.

\end{proof}

\begin{theorem}
Let R be a left coherent ring such that the class of weakly Ding injective modules is closed under extensions. The following statements are equivalent:\\
 1. The character module of every weakly Ding injective $R$-module is a Gorenstein flat right $R$-module.\\
 2. The character module of every Ding injective $R$-module is a Gorenstein flat right $R$-module.\\
 3. $(w\mathcal{DI},\mathcal{GF})$ is a duality pair.\\
 4. the class of weakly Ding injective modules is closed under direct limits.\\
 5. The class of weakly Ding injective modules is covering.\\

 The equivalent statements (1)- (5) also imply the following statement:\\
 6. The class of Gorenstein flat right R-modules is preenveloping.
\end{theorem}

\begin{proof}
 1 $\Leftrightarrow$ 2 is Proposition 11 (because, by Proposition 14, $(^\bot w\mathcal{DI})^\bot =
 w\mathcal{DI}$ in this case).\\
 1 $\Rightarrow$3 is Proposition 8 (since, by Proposition 14,  $(^\bot w\mathcal{DI})^\bot =
 w\mathcal{DI}$ in this case).\\

 3 $\Rightarrow$ 1 by the definition of a duality pair.\\
 3 $\Rightarrow$ 4 is Proposition 9 (since, by Proposition 14, $(^\bot w\mathcal{DI})^\bot =
 w\mathcal{DI}$ in this case).\\

 4 $\Leftrightarrow$ 5 is Corollary 1 (since, by Proposition 14,  $(^\bot w\mathcal{DI})^\bot =
 w\mathcal{DI}$ in this case).\\

 4 $\Rightarrow$ 3. Since $ w\mathcal{DI}$ is the right half of a complete hereditary pair (by
 Propositions 4 and 14), and since it is also closed under direct limits, it is a definable class ([39]). By Lemma 1.1 ([12]), we have that a module
 $M$ is weakly Ding injective if and only if $M^{++}$ is weakly Ding injective.\\
 By Lemma 7, $(M^+)^+$ being weakly Ding injective implies that $M^+$ is Gorenstein flat.\\
 2 $\Rightarrow$ 6 is Theorem 5.
\end{proof}

%\end{proof}

%\begin{acknowledgment}
%\section*{Aknowledgment}
%I would like to thank Dr. Edgar Enochs for numerous advices and
%discussions while writing this paper.
%\end{acknowledgment}

%\newpage
%\bibliographystyle{plain}
%\bibliography{references}

\begin{thebibliography}{1}

\bibitem{bennis:07:stronglygorenstein}
D.~Bennis and N.~Mahdou.
\newblock {Strongly Gorenstein projective, injective, and flat modules}.
\newblock {\em J. Pure Apl. Algebra}, 210:437--445, 2007

\bibitem{BGH}
D. Bravo and J. Gillespie and M. Hovey.
\newblock{The stable module category of a general
 ring}. available at arXiv:1405.5768

\bibitem{CFH}
L. Christensen and A. Frankild and H. Holm.
\newblock{On Gorenstein projective, injec
tive, and flat modules. A functorial description with applications}
\newblock{ \em J. Algebra},
 302(1):231–279, 2006.

\bibitem{chen-zhang}
J. Chen, X. Zhang.
\newblock{Coherent rings and FP-injective rings}
\newblock{Science Press}, 2014.

\bibitem{CH}
L. Christensen, H. Holm.
\newblock The direct limit closure of perfect complexes.
\newblock{ \em Pure
 Appl. Alg.}, 219(3): 449-463, 2015.

\bibitem{mao-ding}
L Mao, and N. Ding,
\newblock Gorenstein FP-injective and Gorenstein flat modules.
\newblock{\em Algebra
 Appl.}, 07(04):491-506, 2008.

\bibitem{MD}
L. Mao, and N. Ding
\newblock{Relative FP-projective modules}.
\newblock{\em Com. Algebra}, 33:1587-1602,
 2005.

\bibitem {DingLiMao}
N. Ding, Y. Li, and L. Mao.
\newblock Strongly Gorenstein Flat Modules.
\newblock{J. Aust. Math.
 Soc.}, 86: 323-338, 2009.

\bibitem{EJT}
E.E. Enochs, O.M.G. Jenda and B. Torrecillas.
\newblock Gorenstein flat modules.
\newblock{\em J.
 Nanjing Univ.} 10-19, 1994.

 \bibitem {Kaplansky}
 E.E. Enochs, J.A. Lopez-Ramos.
 \newblock Kaplansky classess.
 \newblock{\em Rend. Sem. Mat. Univ.
 Padova} (107), 10-19, 2002.

 \bibitem{gorenstein}
 E.E. Enochs and O.M.G. Jenda.
 \newblock Gorenstein injective and projective modules.
 \newblock{\em Math. Zeit}., 220:611- 633, 1995.

 \bibitem{orthogonality}
 E.E. Enochs, O.M.G. Jenda, J. Xu.
 \newblock Orthogonality in the category of complexes.
 \newblock{\em Math. J. Okayama Univ}., 38:25- 46, 1996.

\bibitem{RIP}
S. Estrada, A. Iacob, M. Perez
\newblock Model structures and relative Gorenstein flat
 modules and chain complexes.
 \newblock{\em Contemporary Math}, Volume 751, 135– 176,
 2020.

\bibitem{RHA}
E.E. Enochs and O.M.G. Jenda.
\newblock Relative Homological Algebra.
\newblock{\em Walter de
 Gruyter}, 2000. De Gruyter Exposition in Math.

 \bibitem{eklof}
 P. Eklof.
 \newblock Homological algebra and set theory.
 \newblock{\em Trans. American Math. Soc.},
 (227):207-225,1977.

\bibitem{dinginj}
J. Gillespie.
\newblock On Ding injective, Ding projective and Ding flat modules and
 complexes.
 \newblock{\em Rocky Mtn. J. Math.}, 47(8), 2641=-2673, 2017.

 \bibitem{gao}
 Z. Gao and F. Wang.
 \newblock Coherent rings and Gorenstein FP-injective modules.
 \newblock{\em Com. Algebra}, 40: 1669-1679, 2012.

 \bibitem{gil17}
 J. Gillespie.
 \newblock Gorenstein complexes and recollements from cotorsion pairs.
 \newblock{\em Ad
vances in Math.}, 291: 859- 911, 2016.

 \bibitem{gil18}
 J. Gillespie.
 \newblock Model structures on modules over Ding-Chen rings.
 \newblock{\em Homology,
 Homotopy Appl.}, 12(1): 61-73, 2010.

\bibitem{gil19}
J. Gillespie.
\newblock On Ding injective, Ding projective and Ding at modules and com
plexes.
\newblock{\em Rocky Mountain J. Math.}, 47: 2641-2673, 2017.

 \bibitem{gil.flatmodel}
 J. Gillespie.
 \newblock The flat model structure on Ch(R).
 \newblock{\em Trans. Amer. Math. Soc.}, 356:
 33693- 3390, 2004.

 \bibitem{gil-iacob}
 J. Gillespie, A. Iacob.
 \newblock Duality pairs, generalized Gorenstein modules and Ding
 injective envelopes.
 \newblock{\em Comptes Rendus. Mathmatique}, 360: 381-398, 2022.

 \bibitem{harada}
 M. Harada.
 \newblock Hereditary semi-primary rings and triangular matrix rings.
 \newblock{\em Nagoya
 J. Math.},27(2) :463-484, 1966.

 \bibitem{holgordim}
 H. Holm.
 \newblock Gorenstein homological dimensions.
 \newblock{\em J. Pure and Appl. Alg.}, 189:167
193, 2004.

 \bibitem{cotpairs}
 H.Holm and P. Jørgensen
 \newblock Cotorsion pairs induced by duality pairs.
 \newblock{\em J. Commut.
 Algebra } 1(4): 621-633, 2009.

 \bibitem{purity}
 H.Holm and P. Jørgensen.
 \newblock Covers, preenvelopes, and purity.
 \newblock{\em Illinois Journal
 of Mathematics}, 52(2): 691703, 2008.

 \bibitem{gengor}
 A. Iacob.
 \newblock Generalized Gorenstein modules.
 \newblock{\em Algebra Colloq.}, 29(4): 651-662,
 2022.
 \bibitem{amuc25}
 A. Iacob.
 \newblock Gorenstein injective modules and Enochs conjecture.
 \newblock{\em Acta Mathe
matica Universitatis Comenianae}, 93(4), 197–204, 2025.

 \bibitem{directlimits}
 A. Iacob
 \newblock Direct limits of Gorenstein injective modules.
 \newblock{\em Osaka J. Math.}, to
 appear.
 \bibitem{weakly23}
 A. Iacob.
 \newblock Weakly Ding injective modules and complexes.
 \newblock{\em Com. Algebra}, 51(12),
 4899-4912 , 2023.

 \bibitem{weakly25}
 A. Iacob.
 \newblock Weakly Ding injective preenvelopes and covers.
 \newblock{\em Com. Algebra}, to
 appear, Published online: 17 Apr 2025.


\bibitem{GorFPinj.dim}
J.S. Hu, Y.X. Geng, Z.W. Xie, D.D. Zhang.
\newblock Gorenstein FP-injective dimension
 for complexes.
 \newblock{\em Com. Algebra}, 43(8): 3515-3533, 2015.

 \bibitem{kathy}
 K. Pinzon.
 \newblock Absolutely pure modules. PhD Dissertation.

 \bibitem{stovicek}
 J. Stovicek.
 \newblock On purity and coderived and singularity categories. preprint, available at [arXiv:1412.1615]

 \bibitem{lam}
 T. Y. Lam.
 \newblock Lectures on Modules and Rings. Springer Verlag, 1999.

 \bibitem{murfet}
 D. Murfet and S. Salarian,
 \newblock Totally acyclic complexes over noetherian
 schemes.
 \newblock{\em Adv. Math.} 226 (2011), 1096–1133.

 \bibitem{jrgr}
 J.R. Garc´ ıa Rozas.
 \newblock Covers and evelopes in the category of complexes of modules.
 CRC Press LLC, 1999.

 \bibitem{saroch-st}
 Jan Saroch and Jan Stovcek.
 \newblock Singular compactness and definability for cotor
sion and Gorenstein modules,
\newblock{\em Selecta Math.} (N.S.) 26, 2020.

 \bibitem{saroch}
 J. Saroch and J. Stovicek.
 \newblock The countable Telescope Conjecture for module
 categories.
 \newblock{\em Advances in Math.}, 219(3): 10021036, 2008.

 \bibitem{trlifaj}
 J. Trlifaj.
 \newblock Approximations of modules. Lecture notes for NMAG 31, available
 online at www.karlin.m .cuni.cz/ trlifaj.
 \bibitem{Yang}
 G. Yang and Z. Liu and L.Liang.
 \newblock On Gorenstein at preenvelopes of complexes.
 \newblock{\em
 Rend. Sem. Mat. Univ. Padova}, 129: 171-187, 2013.

 \bibitem{wang.liu}
 Z. Wang, Z. Liu
 \newblock FP-injective complexes and FP-injective dimension of com
 plexes.
 \newblock{\em Aust. Math. Soc.} , 91: 163-187, 2011.

 \bibitem{yang}
 G. Yang, Z.K. Liu, L. Liang.
 \newblock Ding Projective and Ding Injective Modules.
 \newblock{\em Alg.
 Colloquium 20(4)} , 189(1): 601-612, 2013.

 \bibitem{yang.liu11}
 G. Yang and Z. Liu.
 \newblock Strongly cotorsion (torsion-free) modules and cotorsion
 pairs.
 \newblock{\em Proc. Edinburgh Math. Soc.}, 52: 783797, 2011.

 \bibitem{YLL}
 G. Yang, Z.K. Liu, L. Liang.
 \newblock Model Structures on Categories of Complexes
 Over Ding-Chen Rings.
 \newblock{\em Com. Alg.}, 41: 50-69, 2013.

\bibitem{zebg.chen}
Y. Zeng, J. Chen.
\newblock On Gorenstein FP-injective modules.
\newblock{\em Journal of Southeast
 University}, 27(1): 115-118, 2011

 \bibitem{yang2012}
 G. Yang and K. Z. Liu.
 \newblock Gorenstein flat covers over GF-closed rings.
 \newblock{\em Comm.
 Algebra}, 40:1632–1640, 2012.




%\bibitem{asadollahi:16:gorenstein}
%J. Asadollahi and T. Dehghanpour and R. Hafezi.
%\newblock {Existence of Gorenstein projective precovers}.
%\newblock {\em Rend. Sem. Mat. Univ. Padova }, to appear.






\end{thebibliography}

\vspace{11mm}

\textbf{Author}:\\

Alina Iacob, Georgia Southern University, USA, aiacob@georgiasouthern.edu\\

\end{document}